\documentclass[a4paper, 10 pt]{ieeeconf}  

\usepackage{graphicx} 
\usepackage{amsmath} 
\usepackage{amssymb}  
\usepackage{dsfont}
\usepackage{mathtools}

\def\longrightharpoonup{\relbar\joinrel\rightharpoonup}
\def\longleftharpoondown{\leftharpoondown\joinrel\relbar}

\def\longrightleftharpoons{
  \mathop{
    \vcenter{
      \hbox{
      \ooalign{
        \raise1pt\hbox{$\longrightharpoonup\joinrel$}\crcr
	  \lower1pt\hbox{$\longleftharpoondown\joinrel$}
	  }
      }
    }
  }
}

\newtheorem{thm}{Theorem}
\newtheorem{lem}{Lemma}
\newtheorem{newdef}{Definition}

\title{\LARGE \bf Optimization-based Lyapunov function construction for continuous-time Markov chains with affine transition rates}

\author{Andreas Milias-Argeitis, Mustafa Khammash %
\thanks{A. Milias-Argeitis and Mustafa Khammash  are with the
Department of Biosystems Science and Engineering, ETH Zurich, Mattenstrasse 26, 4058, Basel, Switzerland
        }%
}

\begin{document}

\maketitle
\thispagestyle{empty}
\pagestyle{empty}

\begin{abstract}
We address the problem of Lyapunov function construction for a class of continuous-time Markov chains with affine transition rates, typically encountered in stochastic chemical kinetics. Following an optimization approach, we take advantage of existing bounds from the Foster-Lyapunov stability theory to obtain functions that enable us to estimate the region of high stationary probability, as well as provide upper bounds on moments of the chain. Our method can be used to study the stationary behavior of a given chain without resorting to stochastic simulation, in a fast and efficient manner.
\end{abstract}

\section{Introduction}
A classic result in the theory of continuous-time Markov chains (CTMCs) states that an irreducible chain on a countable space has a unique invariant distribution if and only if it is positive recurrent \cite{norris98}. An irreducible and positive recurrent CTMC is commonly called \emph{ergodic}. All finite irreducible Markov chains are automatically ergodic \cite{bremaud}, however verification of ergodicity becomes much harder when the state space is countably infinite. The most common approach to this problem is the use of the so-called Foster-Lyapunov criteria \cite{meyntweedie} which, among others, provide sufficient conditions for positive recurrence. According to one of the main results of this theory, the uniformly negative drift of a suitably defined Lyapunov function outside a finite set of states guarantees that the expected hitting time of this set is finite for any initial condition of the chain, which in turn implies that the chain is ergodic \cite{bremaud}. Despite its elegance, successful application of this result to a given CTMC depends critically on the computation of a Lyapunov function, a non-trivial procedure.

Several attempts have been made to provide guidelines for the construction of Lyapunov functions for specific classes of (discrete-time) Markov chains \cite{fayolle}, yet the general problem of determining a suitable Lyapunov function for a given system remains unsolved. In this work, we propose an optimization-based approach to the computation of Lyapunov functions when the transition rates of a given CTMC are affine functions of the state and the movement of the chain is determined by a finite set of transition vectors. Such characteristics can be found, for example, in models of stochastic chemical kinetics, as well as models used in ecology and epidemiology. Our method relies on the formulation of a semidefinite optimization program (SDP), which can be solved efficiently by existing SDP solvers.

Several classic results in Markov chain theory have demonstrated how Lyapunov functions can be used to provide bounds on stationary expectation of a function of the chain. Our approach enables us to optimize these bounds over a given class of Lyapunov functions. In this way, approximations of quantities related to the stationary behavior of the chain can be obtained fast, without stochastic simulation or solution of the Kolmogorov equations.

The rest of the paper is organized as follows: Section \ref{preliminaries} provides the necessary mathematical background of Foster-Lyapunov theory and the associated bounds that can be obtained using Lyapunov functions. Next, Sections \ref{optim_approach} and \ref{linearlyap} present the main idea behind our approach and its application to Markov chains with affine transition rates. Several examples are considered in Section \ref{examples} to demonstrate the applicability and effectiveness of our method. The conclusions of our study and some current research directions are finally summarized in Section \ref{discussion}.

\section{Foster-Lyapunov ergodicity criterion and associated bounds}\label{preliminaries}
We consider an irreducible CTMC $\{X(t),~t\geq 0\}$ with state space $S\subseteq \mathbb{N}^n_0$ (where $\mathbb{N}_0$ is the set of nonnegative integers). The infinitesimal generator of $X$ is denoted by $Q=(q(x,y))_{x,y\in S}$, where $q(x,y)$ denotes the transition rate from state $x$ to state $y$.  We assume that $Q$ is conservative ($\sum_{y\in S}q(x,y)=0,~\forall x\in S$) and that each state $x$ leads to finitely many states $y$. More specifically, we assume that the movement of the chain is controlled by a finite set of constant \emph{transition vectors} $r_1,~,r_2,\dots,r_m\in\mathbb{N}^n$, so that possible transitions out of $x\in S$ lead into states $y_1=x+r_1,~y_2=x+r_2,\dots,y_m=x+r_m$. Further, we assume that each transition rate $q(x,y)$ is affine in $x$.

In the rest of the paper, we will denote $q(x,x+r_k)$ by $q_k(x)$ and $\sum_{k=1}^mq_k(x)$ by $q(x)$. With this notation, application of the generator $Q$ to a function $V:S\to\mathbb{R}$ yields a new function $QV$, called the \emph{drift} of $V$, which is given by
\begin{align*} QV(x)&=\sum_{k=1}^m q_k(x)V(x+r_k)-q(x)V(x)\\&=\sum_{k=1}^mq_k(x)\left(V(x+r_k)-V(x)\right).\end{align*}

The following theorem provides sufficient conditions for a given irreducible CTMC on $S$ to be ergodic.

\begin{thm}[\cite{tweedie75}]\label{mainth}
Suppose there exists a function $V:S\to \mathbb{R}_+$ and a finite set $C\subset S$ such that\
\begin{align}
&QV(x)\leq -1,~\forall x\in S\setminus C \label{negdrift}\\
&QV(x)<+\infty~\forall x\in C \label{posdrift} \\
&\|x\|\to\infty \implies V(x)\to\infty\mbox{, where }\|x\|=\sum_{i=1}^nx_i
\end{align}
Then the chain is  non-explosive and ergodic.
\end{thm}

Function $V$ is commonly called a \emph{stochastic Lyapunov function}. Note that nonnegativity of $V$ could alternatively be replaced by the condition that $V$ is lower-bounded over $S$ \cite{bremaud}. Conditions (\ref{negdrift}) and (\ref{posdrift}) can be combined in one inequality of the form
\begin{equation}\label{lyapeq}
QV(x)\leq -1+b\mathds{1}_C(x),
\end{equation}
where $\mathds{1}_C(\cdot)$ denotes the indicator function of $C$ and
\[b=\sup_{x\in C}QV(x)+1.\]

The following is a simple consequence of ergodicity:

\begin{lem}
$b\geq 1$ in \eqref{lyapeq}
\end{lem}

\begin{proof}
Since $V$ is bounded from below, $\sup_{x\in S}QV(x)<0$  would imply that that $V(X(t))$ is a supermartingale that converges to a constant \cite{syski92}, and the chain could not be ergodic. Hence, it must hold $\sup_{x\in S}QV(x)\geq 0$, i.e. $b\geq1$.
\end{proof}

Intuitively, Theorem \ref{mainth} states that for an irreducible CTMC the process $V(X(t))$ behaves like a supermartingale outside a finite set $C$, i.e. $V$ decreases on average along the trajectories of the chain until the process hits $C$, the so-called \emph{refuge set}. This interpretation makes an interesting connection with deterministic Lyapunov theory, and could lead one to think that $C$ possesses some special property among all subsets of $S$. However, this is not the case. To see this, we need two hitting time definitions \cite{norris98}:

\begin{newdef}
Given a set $A\subset S$, the first entrance time to $A$ (or the {\bf hitting time} of $A$), denoted by $\tau_A$, is defined as
\[\tau_A=\inf\{t:t\geq 0,X(t)\in A\},\]
with the infimum over an empty set taken to be $\infty$. Using $\tau_{A^c}$, the hitting time to $A^c$, we next define the {\bf first return time} to $A$, denoted by $\sigma_A$:
\[\sigma_A = \inf\{t:t>\tau_{A^c},X(t)\in A\}.\]
\end{newdef}

Given a positive recurrent CTMC, we know that the expected return time to any state, $\mathbb{E}_x[\sigma_x]$ is finite \cite{norris98}. In turn, this implies that $\mathbb{E}_y[\tau_x]$ is also finite for any $x$ and $y$ \footnote{This can be easily seen by defining the taboo transition probability ${}_xP(t,x,y)=\mathbb{P}[X(t)=y|X(0)=x,X(s)\neq x,0<s\leq t]$ and observing that $\mathbb{E}_x[\sigma_x]\geq{}_xP(t,x,y)(\mathbb{E}_y[\tau_x]+t)$. Since the chain is irreducible and $\mathbb{E}_x[\sigma_x]<\infty$, $\mathbb{E}_y[\tau_x]<\infty$ as well.}. Consequently, the expected hitting time, $\mathbb{E}_x[\tau_B]$, of any set $B\subset S$ as a function of $x$ is also finite. This implies that $V_B(x):=\mathbb{E}_x[\tau_B]$ for $B$ finite is a Lyapunov function for the chain, since we have the following

\begin{lem}\label{lemma_return}
The function $V_B(x)$ is the pointwise minimal solution to the system
\begin{align*}
&V_B(x) = 0,~x\in B\\
&QV_B(x)\leq -1,~x\notin B.
\end{align*}
\end{lem}

\begin{proof} Using a first step decomposition and the definition of $V_B$, we see that for $x\notin B$ \[V_B(x)=\sum_{k=1}^m\displaystyle\frac{q_k(x)}{q(x)}\left(\frac{1}{q(x)}+V_B(x+r_k) \right).\]
Application of the generator on $V_B(x)$ then gives \[QV_B(x)=\sum_{k=1}^mq_k(x)V_B(x+r_k)-q(x)V_B(x).\] Substituting the expression for $V_B(x)$ obtained above, we see that $QV_B(x)\leq -1$ if $x\notin B$.

Minimality of $V_B$ follows from Theorem 4.3 of \cite{meyn93}: if \eqref{lyapeq} holds for some $V$ and $C=B$, then
\[\mathbb{E}_x[\tau_B]\leq V(x),~\forall x\notin C.\]
\end{proof}

We thus see that any finite set $B$ can serve as refuge for at least one Lyapunov function. In other words, no set holds a prominent position in (\ref{lyapeq}), in contrast to deterministic Lyapunov theory. Consequently, there is a large freedom in the choice of stochastic Lyapunov functions. This freedom can be exploited, however, since we know that Lyapunov functions can provide bounds to stationary expectations of functions of the chain, as well as bounds on hitting times of sets (an example of the latter was already used in the proof of Lemma \ref{lemma_return}). One can therefore \emph{optimize} over a class of candidate Lyapunov functions to obtain such bounds, and thus gain insight into the stationary behavior of a given chain. Below we present more analytically the bounds we consider in this work.

\subsection{Bounds for stationary set probabilities}\label{probbound}
Consider the problem of finding a set $C$ that contains a large portion of the invariant distribution. Such a set defines a ``central'' region in the state space of the chain, in which it is most probable to find the sample paths at stationarity (notice the similarity of such a set to the equilibrium point of a nonlinear system). A simple calculation based on (\ref{lyapeq}) \cite{meyn93,dayar11} shows that
\begin{equation}\label{setbound}
\pi(C)\geq \frac{1}{b},
\end{equation}
where $\pi$ denotes the invariant distribution of the chain. This is intuitively expected, since $b-1$ indicates the average maximum positive rate of change of $V$ when $X(t)$ is in $C$, and thus quantifies the tendency of the process to move out of $C$. It is no surprise then that $b$ is intimately connected with $\pi(C)$, i.e. the fraction of time the process spends on average inside $C$ at stationarity.

For a given Lyapunov function, one can also reverse the process and find a set $C_\epsilon$ such that $\pi(C_\epsilon)\geq 1-\epsilon$. This will be some super-level set of $QV(x)$. More concretely, given a Lyapunov function $V$ and the associated constant $b$ in (\ref{lyapeq}), we know that $C=\{x:QV(x)\geq -1\}$. Next, given an $\epsilon$, we can find a $\delta\in\mathbb{R}$ such that the new super-level set $C_\epsilon=\{x:QV(x)\geq -1+\delta\}$ has probability greater than $1-\epsilon$. To do this, we first note that $QV(x)\leq b-1~\forall x\in S$ so, using $C_\epsilon$ in place of $C$ we can write
\[QV(x)\leq -1+\delta +(b-1+1-\delta)\mathds{1}_{C_\epsilon}=-1+\delta+(b-\delta)\mathds{1}_{C_\epsilon}.\]
Using Theorem 4.3 of \cite{meyn93}, we know that
\[\pi(C_\epsilon)\geq\frac{1-\delta}{b-\delta},\]
so we can compute the necessary shift $\delta$ from the equation  \[1-\epsilon=\frac{1-\delta}{b-\delta}.\]

\subsection{Moment bounds}\label{momentbounds}

Lyapunov functions can be also used to obtain moment bounds: As shown in \cite{meyn93}, if (\ref{lyapeq}) is generalized to
\begin{equation}\label{generalcondition}
QV(x)\leq -f(x)+b\mathds{1}_C(x),
\end{equation}
where $f\geq 1$, it also holds that
\begin{equation}\label{momentbound}
\pi(f)\leq b.
\end{equation}

From (\ref{generalcondition}) and (\ref{momentbound}) (proven in Theorem 4.3 of \cite{meyn93}), the stationary expectation of a given nonnegative function $f$ can be upper-bounded using the maximum drift obtained from a suitably chosen Lyapunov function $V$.

\section{An optimization approach to Lyapunov function design}\label{optim_approach}
The above observations can serve as starting points towards the construction of appropriate Lyapunov functions for bounding invariant quantities related to a given CTMC. To demonstrate the form of the resulting optimization problems, we consider the case of finding a set that contains a large fraction of the stationary probability mass of the chain. In other words, our goal is to obtain a lower bound on the probability mass of a set $C$, which is described as the (super)level set of the function $QV(x)$. In abstract terms, this can be posed as the following optimization problem:

\begin{equation}\label{lyapopt}
\begin{aligned}
	\underset{V\in\mathcal{V},~b'}{\text{min}}~~&b' \\
	\text{s.t.}~~ &QV\leq b', ~\forall x\in S\\
    &QV\leq-1,~\forall x\notin D\subset S,~D \mbox{ compact}\\
    & V\geq 0\mbox{ and }\|x\|\to\infty\implies V(x)\to\infty
	\end{aligned}
\end{equation}
$\mathcal{V}$ denotes some class of functions over which the optimum is sought. The second constraint requires that $QV$ eventually becomes negative outside a compact set $D$ and is necessary for $V$ to be a Lyapunov function. We then know that the set $C=\{x:QV(x)\geq -1\}$ will lie inside $D$. The intuition behind this formulation is to look for a function $V\in\mathcal{V}$ for which the positive drift over $D$ is minimal. Since the positive drift is directly related to the tendency of the process trajectories to move out of the refuge $C$, one expects that a refuge corresponding to a Lyapunov function with the minimal drift will be located in the region where the stationary density of $X$ tends to be higher.

As we shall see, in the case of affine transition rates \eqref{lyapopt} can be cast in the form of polynomial optimization, provided we treat all functions involved as defined on a continuous space ($\mathbb{R}^n_+$). In particular, \eqref{lyapopt} becomes a semidefinite program (SDP) if we focus our search on quadratic Lyapunov functions.

\section{Quadratic Lyapunov functions for CTMCs with affine transition rates}\label{linearlyap}
In the case of affine transition rates the \emph{drift vector} $d(x)=\sum_{k=1}^m q_k(x)r_k$ can be written as $d(x)=Ax+B$, for some $A\in\mathbb{R}^{n\times n}$ and $B\in\mathbb{R}^n$. For these systems we consider the class of quadratic Lyapunov functions $V=(x-x_0)^TR(x-x_0)$, for some $R>0$ and $x_0\in\mathbb{R}^n$, to be determined via optimization.

To derive the analytic form of the optimization problem, we take a closer look at the action of the generator $Q$ on this class of functions:
\begin{align*}
QV(x) = &\sum_{k=1}^mq_k(x)(V(x+r_k)-V(x))=\\
&=\sum_{k=1}^m q_k(x)((x+r_k)^TR(x+r_k)-\\-&2x_0^TR(x+r_k)-x^TRx+2x_0^TRx)=\\
&=2\sum_{k=1}^mq_k(x)r_k^TR(x-x_0)+\sum_{k=1}^mq_k(x)r_k^TRr_k=\\
&=2(Ax+B)^TR(x-x_0)+\sum_{k=1}^mq_k(x)r_k^TRr_k\\
&=2\left(x^TARx-x^TARx_0+B^TRx-B^TRx_0\right)+\\&+\sum_{k=1}^mq_k(x)r_k^TRr_k=\\
& =x^T(A^TR+RA)x-2x^TARx_0+2B^TRx-\\&-2B^TRx_0+\sum_{k=1}^mq_k(x)r_k^TRr_k.
\end{align*}
The decision variables in this expression are $R$ and $x_0$. Despite the fact that they appear in a product, we observe that $x_0$ always appears in a product with $R$. By defining $y_0=Rx_0$, we can optimize over $R$ and $y_0$, and recover $x_0$ whenever $R$ is invertible. When this condition fails, $V$ can be defined without the constant term $x_0^TRx_0$. This shift has no effect on our results.

We thus see that $QV$ is a quadratic function of the form $f(x)=x^TTx+2u^Tx+\beta$, for which we know that
\[f(x)\geq 0,~\forall x\Leftrightarrow \begin{bmatrix}T&u\\u&\beta\end{bmatrix}\geq 0.\]
Problem \eqref{lyapopt} then takes the form of an SDP, which can compactly be written as

\begin{equation}\label{lyapopt_lin}
\begin{aligned}
	\underset{R,~y_0,~b'}{\text{min}}~~&b' \\
	\text{s.t.}~~ &QV\leq b' ~\forall x\\
    & QV\leq-1,~\forall x\notin D, D \mbox{ compact}\\
    & R\geq 0\\
	\end{aligned}
\end{equation}

The second constraint requires that $QV$ eventually becomes negative outside a compact set $D$ and is necessary because the third constraint alone is not enough to guarantee that numerical solvers will not converge to the trivial solution $V\equiv 0$. In the case of Lyapunov function optimization for linear dynamical systems, where the stability of $x=0$ is typically studied, one can ensure non-degeneracy of solutions by requiring $V-\epsilon \sum_{i=1}^nx_i^2\geq 0$ for a small positive $\epsilon$. This would not be a good choice in our case, as $V$ is not expected to be homogeneous in $x$ (and thus have a minimum at zero), and such a constraint could severely affect the result of the optimization. Our second constraint requires the choice of a given compact set $D$, which can be thought of as an initial guess of where the set $C$ could lie. Provided it is not chosen too small, the optimization outcome will not be affected by the particular choice of $D$.

To maintain the SDP form of the problem, $D$ has to be defined through a set of linear or quadratic inequalities, in which case the second constraint can be written in semi-definite form using the S-procedure \cite{boyd94}.

\subsection{Moment bounds} Using the already established theoretical results presented in Subsection \ref{momentbounds}, we see that we obtaining upper bounds on the stationary mean of a given polynomial function $f\geq 1$ requires solving a problem very similar to  \eqref{lyapopt_lin}. Provided $f$ has a degree $\leq 2$  (i.e., if we seek to bound means and (co)variances), the resulting problem is still an SDP:

\begin{equation}\label{lyapopt_mom}
\begin{aligned}
	\underset{R,~b'}{\text{min}}~~&b' \\
	\text{s.t.}~~ &QV+f\leq b', ~\forall x\\
    & R\geq 0\\
	\end{aligned}
\end{equation}
The non-degeneracy constraint involving the set $D$ is no longer needed: such solutions are no longer possible, thanks to the presence of $f$ in the inequality.

\subsection{Nonlinear transition rates}
When a system contains quadratic or bilinear transition rates (such as in the case of bimolecular reactions in chemical kinetics), using a (general) quadratic Lyapunov function will result in third order polynomials in $QV$. However, when only a few transitions have this feature, one could still search for a quadratic Lyapunonv function by requiring that $Rr_{b}=0$, where $r_b,~b=1,\dots$ are the transition vectors corresponding to these transitions. There cannot be too many such transitions (in comparison to the state size, $n$), otherwise the only feasible solution will be $R=0$. Every constraint of the form $Rr_b=0$ restricts $V$ to be constant along the direction $r_b$, a severe restriction on the shape of the function. Note also that such a $V$ can no longer be positive definite, however it can still be lower-bounded on $\mathbb{N}^n_0$ and satisfy
\[ \|x\|\to\infty \implies V(x)\to\infty,~\mbox{when } x\geq 0,\]
which is enough for Theorem \ref{mainth} to hold (note that our chains evolve on $\mathbb{N}^n_0$). In turn, this is only possible if no $r_b$ has all its components nonnegative (otherwise $V$ would be zero along a direction inside the positive orthant). If the $r_b$'s arise from chemical reaction stoichiometries, they automatically have this property due to mass conservation: a bimolecular reaction must consume some reactants to generate the products.

\section{Examples}\label{examples}
\subsection{A simple gene expression model}
Consider the following reaction scheme corresponding to a simple transcription-translation model:
\begin{align*}
&\varnothing \xrightarrow{100} mRNA\xrightarrow{1}\varnothing\\
&\varnothing \xrightarrow{mRNA} Protein\xrightarrow{0.1}\varnothing\\
\end{align*}
The numbers of mRNA and protein molecules at time $t$ will be denoted by $M(t)$ and $P(t)$ respectively. This system can be modeled as a CTMC \cite{mazza14} $\{\left(M(t),P(t)\right),~t\geq 0\}$ on $\mathbb{N}^2_0$. At each state $(m,p)$, four transitions are possible; their transition vectors are $r_1=[1~0]^T$, $r_2=[-1~0]^T$, $r_3=[0~1]^T$ and $r_4=[0~-1]^T$, with corresponding transition rates $q_1=100$, $q_2=m$, $q_3=m$ and $q_4=0.1p$.

We are first going to search for a Lyapunov function that will provide us with the high-stationary density region of the system. To that end, we solve \eqref{lyapopt_lin} to obtain $b'$, $R$ and $y_0$, for $D=\{(m,p):(m-100)^2+(p-1000)^2\leq 10^5\}$ \footnote{Almost identical results are obtained with $D=\{(m,p):0\leq m\leq 10^4,~0\leq p\leq 10^4\}$ and $D=\{(m,p):m^2+p^2\leq 10^5\}$}. The optimal solution turns out to be
\[\begin{split}&R=\begin{bmatrix}0.0381&0.0096\\0.0096&0.0155\end{bmatrix}\cdot 10^{-11},~x_0=\begin{bmatrix}99.1\\1000.1\end{bmatrix},\\&b'=1.3\cdot 10^{-9}.\end{split}\]
Notice that $V=(x-x_0)^TR(x-x_0)$ is centered very close to the mean of the system ($\begin{bmatrix}100&1000\end{bmatrix}^T$). Following the method of Subsection \ref{probbound}, we can compute the level set of $QV$ which contains more than $\alpha$\% of the stationary mass of the chain. A few of those sets are displayed on Fig. \ref{mrna-prot} below, together with a logarithmic plot of the actual invariant distribution of the system, obtained from SSA. The innermost contour (corresponding to a lower bound of 80\%) encloses 99\% of the invariant probability.

\begin{figure}[h!tb]
\centering
  \includegraphics[width=0.9\columnwidth]{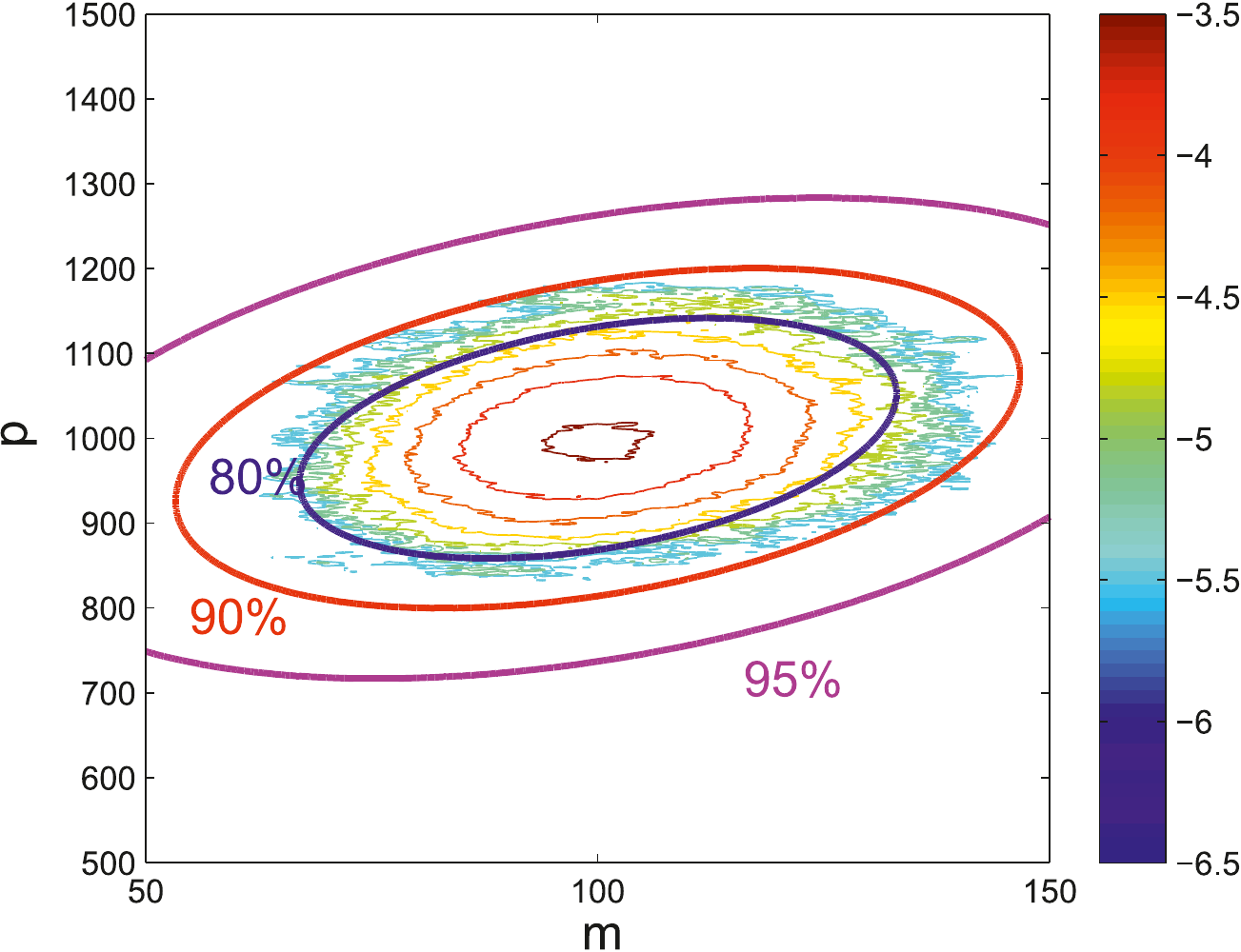}
  \caption{Contour lines of $QV$ overlayed to the logarithm (base 10) of the invariant distribution of the system, obtained from stochastic simulation. The percentages next to each line denote the optimization-based lower bound on the stationary mass enclosed by it.}
  \label{mrna-prot}
\end{figure}

While the lower bounds obtained are certainly conservative, we are nevertheless able to capture the region of high stationary density of the system with relatively good accuracy. On the other hand, this choice results in a very simple SDP that can be solved in 0.38 seconds using YALMIP \cite{YALMIP} with the SDPT3 solver in Matlab.

Turning to moment upper bounds, solution of problem \eqref{lyapopt_mom} for various choices of $f$ are presented in Table \ref{moments_mrna-prot} below.

\begin{table}[h!]
\centering
\caption{Upper moment bounds obtained from optimization problem \eqref{lyapopt_mom}, compared against their true values.}
\label{moments_mrna-prot}
  \begin{tabular}{| c | c | c |}
    \hline
    Function & Upper bound & Actual value \\ \hline
    $f=m$ & $100+9\cdot 10^{-5}$  & 100 \\ \hline
    $f=p$ & $1000+1\cdot 10^{-3}$ & 1000 \\ \hline
    $f=m^2$ & $10100$ & $10100$ \\ \hline
    $f=p^2$ & $1.002\cdot 10^6$ & $1.0019\cdot 10^6$ \\ \hline
    $f=m\cdot p$ & $1.002\cdot 10^5$ & $1.0009\cdot 10^5$ \\ \hline
  \end{tabular}
\end{table}

\subsection{A linear system with three species}
The system is described by the following reactions:
\begin{align*}
&\varnothing \xrightarrow{10} S_1\xrightarrow{0.1}\varnothing\\
&S_1 \xrightarrow{10} S_2\xrightarrow{0.1}\varnothing\\
&S_2 \xrightarrow{20} S_3\xrightarrow{0.1}\varnothing\\
&S_3 \xrightarrow{30} S_1\\
\end{align*}
We denote by $X_1(t)$, $X_2(t)$ and $X_3(t)$ the abundance at time $t$ of $S_1$, $S_2$ and $S_3$ respectively. The reader should hopefully be able to ``translate'' the reactions above into the corresponding transition vectors and rates, based on the presentation of the previous example.

Again, we first determine the region where the stationary density of this system is concentrated. Solution of \eqref{lyapopt_lin} with $D=\{(x_1,x_2,x_3):0\leq x_1\leq 1000,0\leq x_2\leq 1000,0 \leq x_3\leq 1000\}$ provides us with
\[R=\begin{bmatrix}0.225&0.221&0.222\\0.221&0.223&0.221\\0.222&0.221&0.224\end{bmatrix},~x_0=\begin{bmatrix}54.48\\25.9\\17.45\end{bmatrix},~b'=10^{-4}.\]
Again, we notice that $V=(x-x_0)^TR(x-x_0)$ is centered very close to the mean of the system ($\begin{bmatrix}54.7&27.21&18.08\end{bmatrix}^T$). The level set of $QV$ which contains more than 90\% of the stationary mass of the chain is the interior of the contour surface shown on Figure \ref{3dlin}.

\begin{figure}[h!tb]
\centering
  \includegraphics[width=0.9\columnwidth]{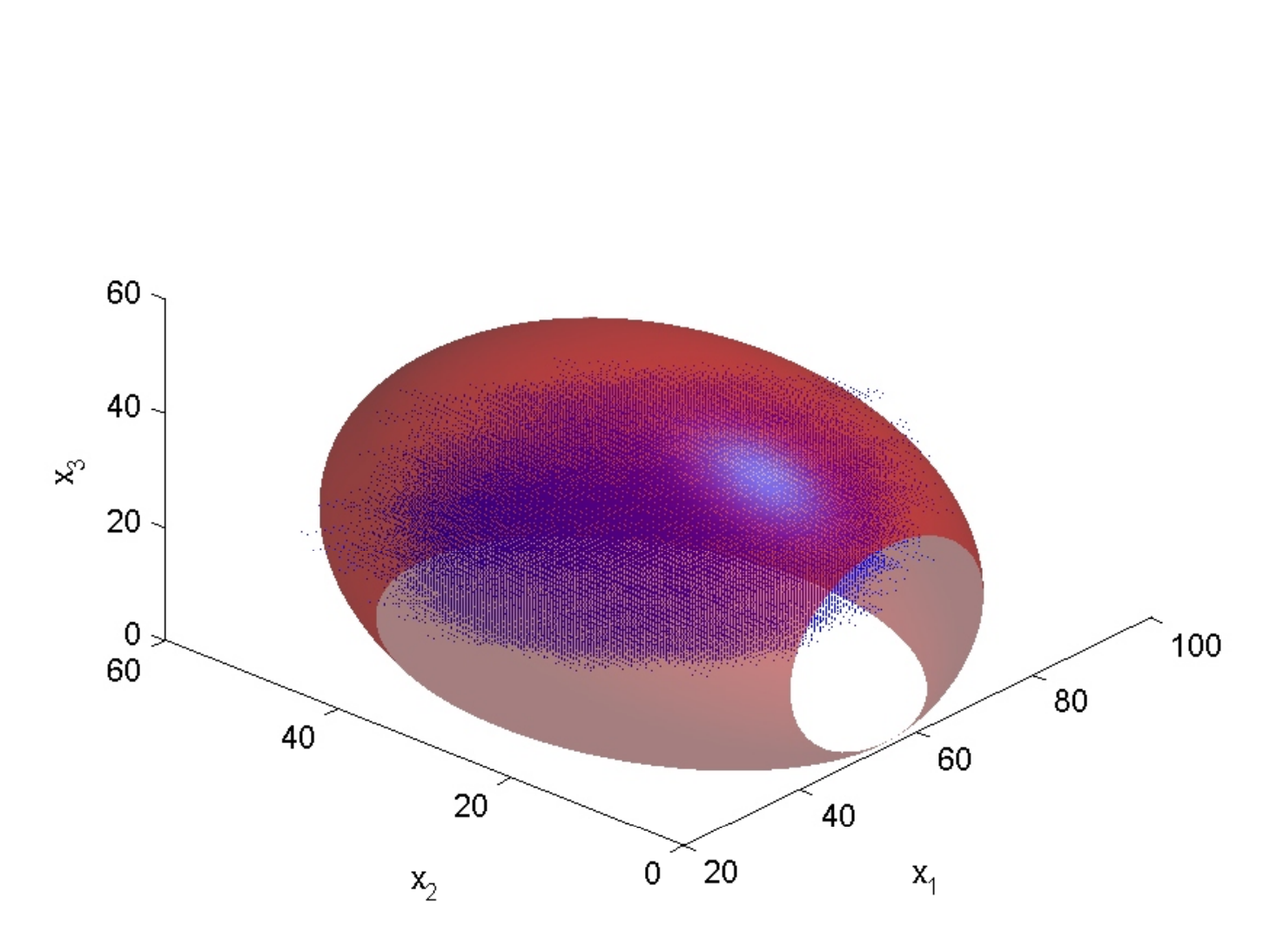}
  \caption{Contour surface of $QV$ corresponding to a 90\% lower bound probability. The actual stationary mass contained in the set is about 99\%. Blue dots mark the points visited by a long chain sample path at stationarity.}
  \label{3dlin}
\end{figure}

To demonstrate the relation of the level set size to the actual volume that the chain explores, the points that are visited by a long sample of the chain are also marked in the plot, as it is impossible to display graphically the contours of the stationary distribution in 3-D. We observe that the high-density region is captured quite well. Optimization took 0.5 sec with YALMIP and SDPT3.

The upper bounds for first- and second-order moments are displayed on Table \ref{moments_3d}.
\begin{table}[h!]
\centering
\caption{Upper moment bounds obtained from optimization problem \eqref{lyapopt_mom}, compared against their true values.}
\label{moments_3d}
  \begin{tabular}{| c | c | c |}
    \hline
    Function & Upper bound & Actual value \\ \hline
    $f=x_1$ & $54.70$  & 54.70 \\ \hline
    $f=x_2$ & $27.21$ & 27.21 \\ \hline
    $f=x_3$ & $18.08$ & $18.08$ \\ \hline
    $f=x_1^2$ & $3047.01$ & $3047.01$ \\ \hline
    $f=x_2^2\cdot p$ & $767.86$ & $767.86$ \\ \hline
    $f=x_3^2\cdot p$ & $345.07$ & $345.07$ \\ \hline
    $f=x_1\cdot x_2$ & $1506.85$ & $1488.71$ \\ \hline
    $f=x_1\cdot x_3$ & $1003.67$ & $989.17$ \\ \hline
  \end{tabular}

\end{table}

\subsection{A nonlinear system with three species}
Our final example is a system described by the following reaction scheme:
\begin{align*}
&\varnothing \xrightarrow{10} S_1\xrightarrow{1}\varnothing\\
&\varnothing \xrightarrow{10} S_2\xrightarrow{1}\varnothing\\
&S_1+S_2 \xrightarrow{1} S_3\xrightarrow{1}\varnothing,\\
\end{align*}
with $X$, $Y$ and $Z$ denoting the abundance of $S_1$, $S_2$ and $S_3$ respectively. Due to the presence of the bimolecular reaction, candidate quadratic Lyapunov functions for this system must satisfy ${R\cdot\begin{bmatrix}-1&-1&1\end{bmatrix}^T=0}$ for the resulting optimization problem to remain in SDP form. Under this constraint, the Lyapunov function that optimally determines the region of maximum stationary density is given by $V=x^TRx-2x^TRx_0$, where
\[R=\begin{bmatrix}0.26&-0.09&0.17\\-0.09&0.26&0.17\\0.17&0.17&0.34\end{bmatrix}\cdot 10^{-2},~Rx_0=\begin{bmatrix}0.016\\0.016\\0.033\end{bmatrix}.\]
Because $R$ is singular, $x_0$ cannot be determined separately. This poses no problem for our approach, as we explained in Section \ref{optim_approach}. The level set of $QV$ corresponding to a lower bound of 90\% is displayed on Figure \ref{nonl}. Optimization took 0.4 seconds in YALMIP with the SDPT3 solver.
\begin{figure}[h!tb]
\centering
  \includegraphics[width=0.9\columnwidth]{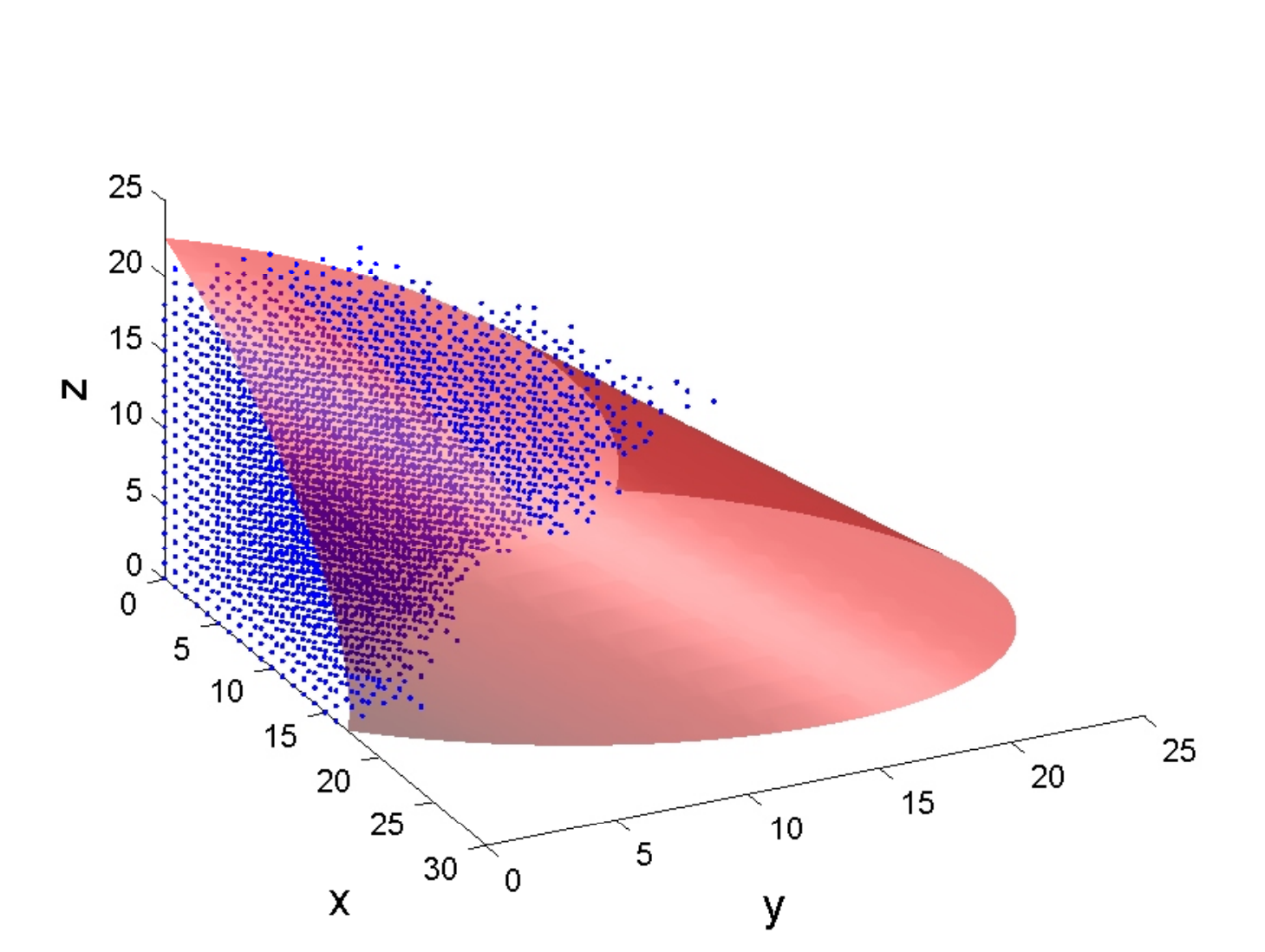}
  \caption{Contour surface of $QV$ corresponding to a 90\% lower bound probability. The actual stationary mass contained in the set is about 99.9\%. Blue dots mark the points visited by a long chain sample path.}
  \label{nonl}
\end{figure}

\section{Discussion \& Outlook}\label{discussion}
Apart from proving stability, Lyapunov functions for deterministic systems are useful in many different ways, for example in estimating the region of attraction of a given equilibrium point or providing convergence rates. In a similar fashion, the well-established theory of stochastic Lyapunov functions for Markov chains has turned them into useful tools for probing the stationary system behavior, besides determining ergodicity. In this work, we have presented an optimization-based approach to Lyapunov function design for locating regions of high stationary probability and bounding moments of CTMCs.

We have shown that even simple quadratic Lyapunov functions can capture a lot about the stationary behavior of a CTMC, and thus provide information about a system without the need for stochastic simulation. The results of this analysis can be used, for example, to determine suitably small truncations of the state space, on which approximate solution methods such as the Finite State Projection algorithm \cite{munsky06} can be applied.

Previous work has employed linear Lyapunov functions to determine exponential ergodicity of CTMCs, and shown that the search of such functions can often be reduced to the solution of a linear program \cite{briat13}. When the main goal is to determine system stability, linear Lyapunov functions can provide the simplest and most efficiently computable certificates. However, when additional system properties are of interest, linear functions are not sufficiently flexible to provide useful answers. For example, estimates of high-density regions cannot be tight enough using a linear Lyapunov function $V$, as the level sets of $V$ over the positive orthant are polyhedra with $n+1$ faces, $n$ of which lie on the coordinate axes. On the other hand, the level sets of quadratic function can be centered away from the origin, and thus provide much tighter estimates of this type.

Due to space limitations, we chose to present the main ideas of our approach only for chains with affine transition rates. One could argue that such systems can be easily studied using moment equations or even the closed form of their probability density, which is available in several -- but not all -- cases. We believe that one can take advantage of this simplicity to check the soundness of a new approach, before moving on to more complex systems with general polynomial transition rates. The presence of such rates leads to general polynomial optimization problems, that can be solved using sum-of-squares (SOS) relaxations \cite{parrilo00}, in the same spirit that polynomial Lyapunov functions are used to study the stability of polynomial dynamical systems \cite{papachristodoulou02}. While the main ideas of our approach remain the same in that case as well, the optimization problem setup becomes a bit more intricate, and will therefore be the topic of a future publication.

\bibliographystyle{IEEEtran}

\end{document}